\documentclass[12pt,twoside, reqno]{amsart}
\usepackage{amsmath}
\usepackage{amssymb}
\usepackage{graphicx}
\usepackage{mathrsfs}
\usepackage{pb-diagram}

\def\stretchx{\Bumpeq{\!\!\!\!\!\!\!\!{\longrightarrow}}}

\newtheorem{theorem}{Theorem}[section]
\newtheorem{lemma}{Lemma}[section]
\newtheorem{definition}{Definition}[section]

\numberwithin{equation}{section}

\DeclareMathOperator{\dist}{dist}
\DeclareMathOperator{\rot}{rot}
\DeclareMathOperator{\Id}{Id}

\begin{document}

\title[Chaotic dynamics]
{Chaotic dynamics in the Volterra predator$-$prey model via linked
twist maps} {
\thanks{{This work has been supported by the MIUR project
``Equazioni Differenziali Ordinarie e Applicazioni".}}
\author[M. Pireddu, F. Zanolin]
{Marina Pireddu and  Fabio Zanolin  }
\address{
{M. Pireddu}, Department of Mathematics
and Computer Science, University of Udine, via delle Scienze 206,
I--33100 Udine, Italy
} \email{marina.pireddu@dimi.uniud.it}
\address{{F. Zanolin}, Department of Mathematics
and Computer Science, University of Udine, via delle Scienze 206,
I--33100 Udine, Italy}
\email{fabio.zanolin@dimi.uniud.it}

\maketitle

\begin{center}
\textit{``Dedicated to the memory of Professor Andrzej Lasota'' }
\end{center}

\noindent
\begin{abstract}
\noindent We prove the existence of infinitely many periodic solutions and
complicated dynamics, due to the presence of a topological horseshoe,
for the classical Volterra predator--prey model with a periodic harvesting.
The proof relies on some recent results about chaotic planar maps
combined with the study of geometric features which are typical
of linked twist maps.
\end{abstract}

\bigskip



\bigskip

\section{Introduction and main results}\label{sec-1}
The classical Volterra predator--prey model concerns the first order planar differential system
\begin{equation*}
\left\{
\begin{array}{ll}
x'= x( a - by)\\
y'= y(-c + dx),\\
\end{array}
\right.
\eqno{(E_0)}
\end{equation*}
where
$a, \;b, \;c, \;d \, > 0$
are constant coefficients. The study of system $(E_0)$ is confined to the open first quadrant
$({\mathbb R}^+_0)^2$
of the plane, since
$x(t) > 0$ and $y(t) > 0$ represent the size (number of individuals or density) of the
prey and the predator populations, respectively.
Such model was proposed by Vito Volterra in 1926 answering D'Ancona's
question about the percentage of selachians and food fish caught in the northern Adriatic
Sea during a range of years covering the period of the World War I
(see \cite{Br-93,Ma-04} for a more detailed historical account).
\\
System $(E_0)$ is conservative and its phase--portrait is that of a global center at the
point
\begin{equation}\label{eq-p0}
P_0\,:= \left( \frac{c}{d},\frac{a}{b}\right),
\end{equation}
surrounded by periodic orbits (run in the counterclockwise sense) which are the level lines of the first integral
\begin{equation}\label{eq-e0}
{\mathcal E}_0(x,y):= d x - c \log x + b y - a \log y,
\end{equation}
that we'll call ``energy'' in analogy to mechanical systems.
The choice of the sign in the definition of the first integral implies that
${\mathcal E}_0(x,y)$ achieves a strict absolute minimum at the point $P_0\,.$
\\
According to Volterra's analysis of $(E_0),$ the average of a periodic solution
$(x(t),y(t)),$ evaluated over a time--interval corresponding to its natural period,
coincides with the coordinates of the point $P_0\,.$

In order to include the effects of fishing in the model,
one can suppose that, during the harvesting time, both the prey and
the predator populations are
reduced at a rate proportional to the size of the population itself. This
assumption leads to the new system
\begin{equation*}
\left\{
\begin{array}{ll}
x'= x( a_{\mu} - by)\\
y'= y(-c_{\mu} + dx),\\
\end{array}
\right.
\eqno{(E_{\mu})}
\end{equation*}
where
$$a_{\mu}:= a - \mu\quad\mbox{and } \; c_{\mu}:= c + \mu$$
are the modified growth coefficients which take into account the fishing rates $-\mu x(t)$ and
$- \mu y(t),$ respectively. The parameter $\mu$ is assumed to be positive but small enough
$(\mu < a$) in order to prevent the extinction of the populations. System $(E_{\mu})$
has the same form like $(E_0)$ and therefore its phase--portrait is that of a global center
at
\begin{equation}\label{eq-pmu}
P_{\mu}:= \left( \frac{c + \mu}{d},\frac{a - \mu}{b}\right).
\end{equation}
The periodic orbits surrounding $P_{\mu}$ are the level lines of the first integral
\begin{equation}\label{eq-emu}
{\mathcal E}_{\mu}(x,y):= d x - c_{\mu} \log x + b y - a_{\mu} \log y.
\end{equation}
The coordinates of $P_{\mu}$ coincide with the average values of the prey and the predator populations
under the effect of fishing (see Figure 1). A comparison between the coordinates
of $P_0$ and $P_{\mu}$  motivates the
conclusion (\textit{Volterra's principle})
that a moderate harvesting has a favorable effect for the prey population
\cite{Br-93}.

\begin{figure}[ht]\label{fig-1}
\centering
\includegraphics[scale=0.28]{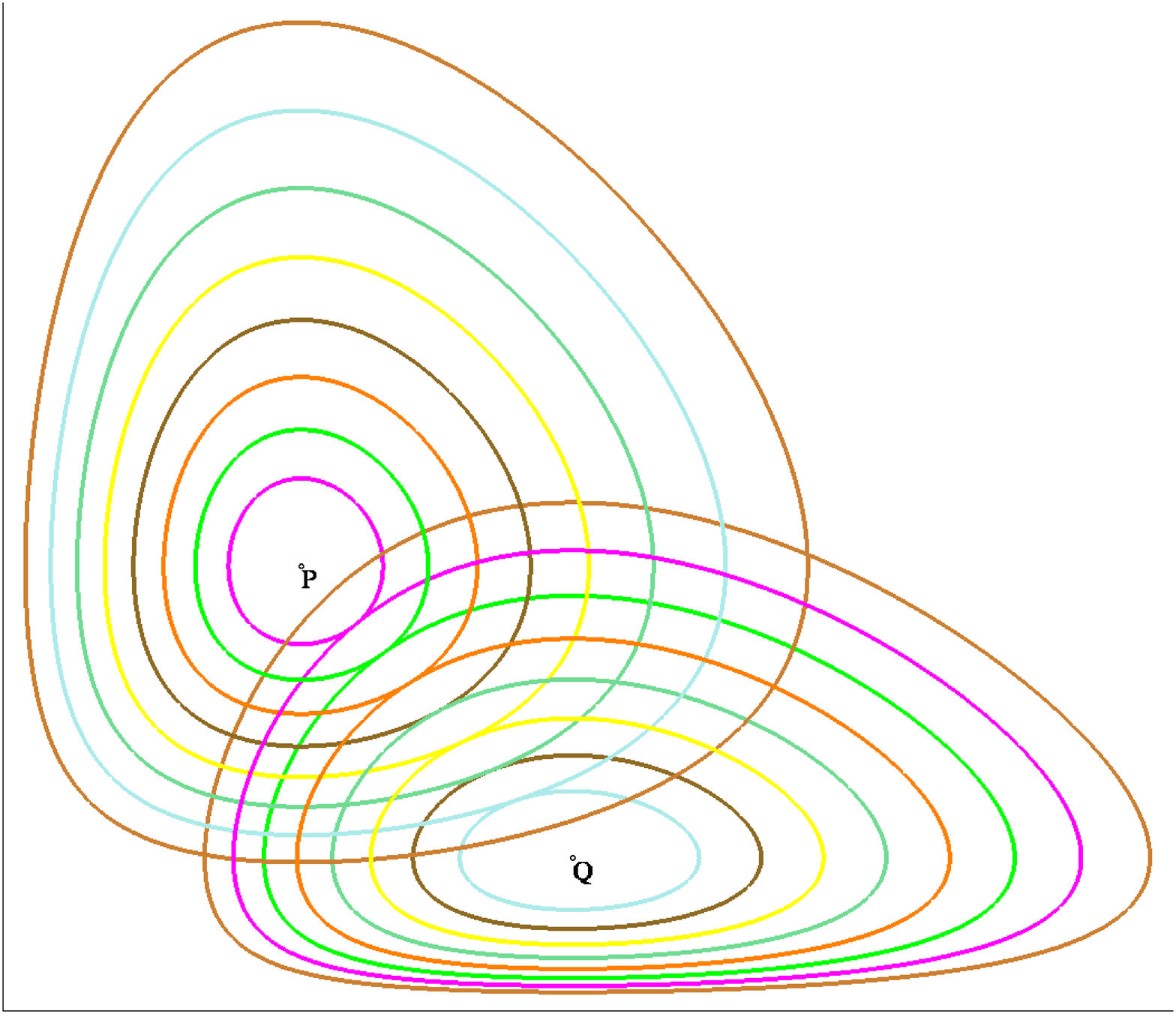}
\caption{\footnotesize In this picture we show some periodic orbits of
the Volterra system $(E_0)$ with center at $P=P_0$
as well as of the
perturbed system $(E_{\mu})$ with center at $Q= P_{\mu}\,$
(for a certain $\mu \in \,]0,a[\,$).}
\end{figure}

\medskip

The Volterra's system (also called Lotka--Volterra's with reference to the work of
Alfred Lotka, who first in 1920 used the same system as a description for certain chemical reactions)
has been sometimes criticized by ecologists and biologists, who refused to consider such a model
as accurate (see \cite{Br-93} for a discussion about this topic), and, in the course of the years,
many variants of it have been proposed. Volterra's himself modified system
$(E_0)$ in \cite{Vob-31,Voa-31}, by replacing the Malthusian growth rates with logistic terms of Verhulst type (see also \cite{Ma-04}).
In order to incorporate the effects of a cyclic environment, periodic coefficients have been
introduced both in the basic model and in its variants.
The past forty years have witnessed a growing interest in such kind of models and several results
have been obtained about the existence, multiplicity and stability of periodic solutions
for Lotka--Volterra type predator--prey models with periodic coefficients
\cite{AmOr-94,Ba-81,Ba-99,BuHu-91,BuFr-81,Cu-77,DiHuZa-95,DiZa-93,DiZa-96,Ha-82,HaMa-87,HaMa-91,Li-95,LGOrTi-96,Or-98,Ta-87}.

\medskip

If we take the Volterra's original model $(E_0)$ and assume a seasonal effect on the coefficients,
we are led to consider a new system of the form
\begin{equation*}
\left\{
\begin{array}{ll}
x'= x( a(t) - b(t) y)\\
y'= y(-c(t) + d(t) x),\\
\end{array}
\right.
\eqno{(E)}
\end{equation*}
where $a(\cdot), b(\cdot), c(\cdot), d(\cdot): {\mathbb R}\to {\mathbb R}$
are periodic functions with a common period $T > 0.$
In such a framework, it is natural to look for harmonic (i.e. $T$--periodic)
or $m$-th order subharmonic solutions (i.e. $mT$--periodic, for some integer $m\geq 2,$ with
$mT$ the minimal period in the set $\{jT: j=1,2,\dots\}$) having
range in the open first quadrant (\textit{positive solutions}).
With this respect, we have the following theorem which
can be derived as a corollary of some results in \cite{DiZa-96} dealing with certain classes
of time--periodic Kolmogorov systems. In Theorem \ref{th-1.1} below,
as well as in the other results of this paper, solutions are meant in the
Carath\'{e}odory sense, that is, $(x(t),y(t))$ is absolutely continuous and satisfies system $(E)$
for almost every $t \in {\mathbb R}.$
Of course, such solutions are of class $C^1$ if the coefficients are continuous.

\begin{theorem}\label{th-1.1}
Suppose that $b(\cdot)$ and $d(\cdot)$ are continuous functions such that
$$b(t)>0,\; d(t) > 0,\quad\forall\, t\in [0,T]$$
and let $a(\cdot), c(\cdot) \in L^1([0,T])$ be such that
$$\bar{a} := \frac{1}{T}\int_0^T a(t)\,dt > 0,\quad \bar{c} := \frac{1}{T}\int_0^T c(t)\,dt > 0.$$
Then the following conclusions hold:
\begin{itemize}
\item[$(e_1)\;$] System $(E)$ has at least one positive $T$--periodic solution\,;
\item[$(e_2)\;$] There exists an index $m^*\geq 2$ such that, for every $m\geq m^*,$  there are at least
two subharmonic solutions $({\tilde{x}}_{1,m},{\tilde{y}}_{1,m})$ and $({\tilde{x}}_{2,m},{\tilde{y}}_{2,m})$
of order $m$ for $(E)$ which do not belong to the same periodicity class and satisfy
$$\begin{array}{lll}
\lim_{m\to \infty}{(\min {\tilde{x}}_{i,m})} = \lim_{m\to \infty}{(\min {\tilde{y}}_{i,m})} = 0\\
\lim_{m\to \infty}{(\max {\tilde{x}}_{i,m})} = \lim_{m\to \infty}{(\max {\tilde{y}}_{i,m})} = + \infty\\
\end{array}\qquad (i=1,2).$$
\end{itemize}
\end{theorem}

\bigskip

\noindent
For a proof and other details, see \cite{DiZa-96}.
We refer to \cite{DiZa-93} for detailed information about the subharmonic solutions of system $(E)$
and to \cite{AmOr-94,Li-95,Or-98} for results about the stability and the number of solutions.
See also \cite{Ba-99,DiHuZa-95,DiZa-96} for more general conditions on the coefficients
ensuring a priori bounds and existence of $T$--periodic positive solutions.

\bigskip
Let us come back for a moment to the original Volterra system with constant coefficients and
suppose that the interaction between the two populations
is governed by system $(E_0)$ for a certain period
of the season (corresponding to a time--interval of length $r_0$)
and by system $(E_{\mu})$ for the rest of the time (corresponding to a time--interval of length $r_{\mu}$).
Assume also that such alternation between $(E_0)$ and $(E_{\mu})$ occurs in a
periodic fashion, so that
$$T:=r_0 + r_{\mu}$$
is the period of the season. In other terms, we consider at first system $(E_0)$ for $t\in [0,r_0[\,.$ Next we switch
to system $(E_{\mu})$ at time $r_0$ and assume that $(E_{\mu})$
rules the dynamics for $t\in [r_0,T[\,.$ Finally, we suppose that
we switch again to system $(E_0)$ at time $t=T$ and repeat the cycle with $T$--periodicity.

Such two--state alternating behavior can be
equivalently described in terms of equation $(E),$ by assuming
$$
a(t)= {\hat{a}}_{\mu}(t):= \left\{
\begin{array}{llll}
a\,\quad &\mbox{for } \, 0\leq t < r_0\\
a - \mu\,\quad &\mbox{for } \, r_0\leq t < T\,,
\end{array}
\right.
$$
$$
c(t)={\hat{c}}_{\mu}(t):= \left\{
\begin{array}{llll}
c\,\quad &\mbox{for } \, 0\leq t < r_0\\
c + \mu\,\quad &\mbox{for } \, r_0\leq t < T\,,
\end{array}
\right.
$$
as well as
$$b(t)\equiv b,\;\; d(t)\equiv d,$$
with $a,b,c,d$ positive constants and $\mu$ a parameter with $0 < \mu < a.$
Hence we can now consider system
\begin{equation*}
\left\{
\begin{array}{ll}
x'= x( {\hat{a}}_{\mu}(t) - b y)\\
y'= y(-{\hat{c}}_{\mu}(t) + d x),\\
\end{array}
\right.
\eqno{(E^*)}
\end{equation*}
where the piecewise constant functions ${\hat{a}}_{\mu}$ and ${\hat{c}}_{\mu}$ are supposed to be
extended to the whole real line by $T$--periodicity.
Clearly, Theorem \ref{th-1.1} holds for equation $(E^*),$ ensuring the existence of at
least one positive $T$--periodic solution and $m$-th order subharmonics of any sufficiently large order.

\medskip

It is our aim now to prove that system $(E^*)$ generates far richer dynamics. Indeed, we'll show the
presence of a \textit{topological horseshoe} for the Poincar\'{e} map
$$\phi: ({\mathbb R}^+_0)^2 \to ({\mathbb R}^+_0)^2,\quad
\phi(z):= \zeta(T,z),$$
where $\zeta(\cdot,z) = (x(\cdot,z),y(\cdot,z))$ is the solution of $(E^*)$
starting from $z = (x_0,y_0)\in ({\mathbb R}^+_0)^2$ at the time $t=0.$
As a consequence, all the
complexity which is associated to an horseshoe geometry (like, for instance,
a semiconjugation to the Bernoulli shift, sensitivity with respect to initial conditions,
positive topological entropy, a compact set containing a dense subset of periodic points)
will be guaranteed.
To this goal, we apply recent developments \cite{PaPiZa->} which connect the analysis of
certain planar ODEs to the theory of \textit{linked twist maps}.
With such a term, one usually designates some geometric configurations characterized by
the alternation of two planar homeomorphisms (or diffeomorphisms) which twist two overlapping
annuli. More precisely, we have two annular regions ${\mathcal A}$ and ${\mathcal B}$
which cross in two disjoint topological rectangles ${\mathcal R}_1$ and
${\mathcal R}_2\,.$ Each annulus is turned onto itself by a homeomorphism
which leaves the boundaries of the annulus invariant.
Both the maps act in their domain so that a twist effect is produced.
This happens, for instance, when the angular speed is monotone with respect
to the radius. Under certain assumptions,
it is possible to prove the existence of a Smale horseshoe inside ${\mathcal R}_{i}$
($i=1,2)$ \cite{De-78}.
Linked twist maps (LTMs) were originally studied in the 80s by
Devaney \cite{De-78}, Burton and Easton \cite{BuEa-80} and Przytycki \cite{Pr-83, Pr-86}.
As observed in \cite{De-78},
such maps naturally appear in mathematical models for particle motions in a magnetic field
and in differential geometry. Geometrical configurations
related to LTMs can be also found in the restricted three--body problem
\cite[pp. 90--94]{Mo-73}.
In more recent years significant applications of LTMs have been performed in the area of
fluid mixing (see, for instance, \cite{StOtWi-06, Wi-99, WiOt-04}).

\medskip

With the aid of Figure 2 we try now to explain how to show the
presence of a horseshoe--type geometry for the switching system
$(E^*).$ As a first step, we take two closed overlapping annuli
made up by level lines of the first integrals associated to system
$(E_0)$ and $(E_{\mu}),$ respectively. In particular, the inner
and outer boundaries of each annulus are closed trajectories
surrounding the equilibrium point ($P=P_0$ for system $(E_0)$ and
$Q=P_{\mu}$ for system $(E_{\mu})$). Such annuli (that we call
from now on ${\mathcal A}_P$ and ${\mathcal A}_Q$) intersect in
two compact disjoint rectangular sets ${\mathcal R}_1$ and
${\mathcal R}_2\,.$ The order in which we decide to name the two  regions
${\mathcal R}_1$ and ${\mathcal R}_2$ is completely arbitrary.
Whenever we enter in a
setting like that described in Figure 2, we say that the annuli
${\mathcal A}_P$ and ${\mathcal A}_Q$ are \textit{linked
together}. Technical conditions on the energy level lines defining
${\mathcal A}_P$ and ${\mathcal A}_Q\,,$ which ensure the linking
condition, are presented in Section \ref{sec-2}.

\begin{figure}[ht]\label{fig-2}
\centering
\includegraphics[scale=0.28]{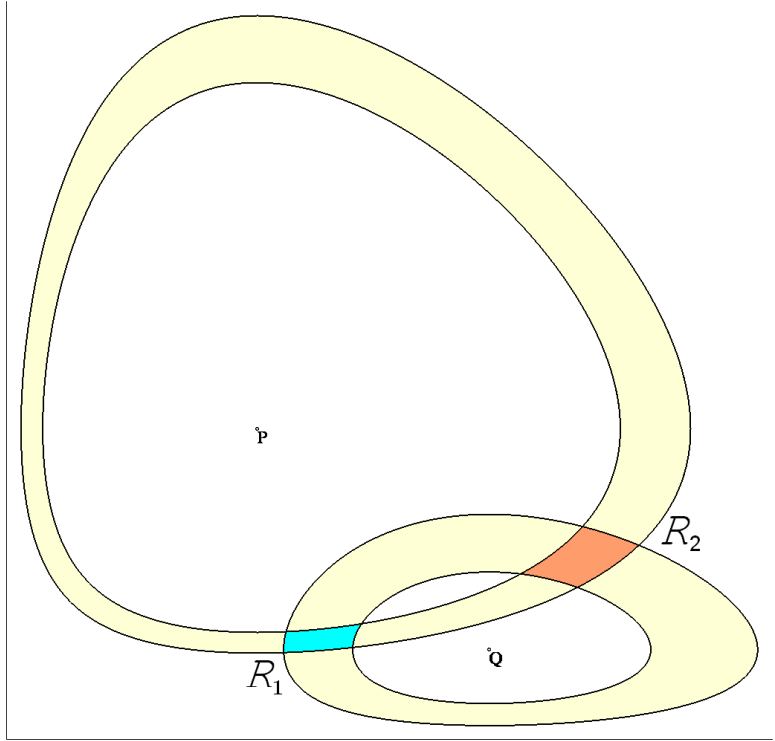}
\caption{\footnotesize The two annular regions $\mathcal A_P$ and
$\mathcal A_Q$ (centered at $P$ and $Q,$ respectively)  are linked
together. We have drawn with a darker color the two
rectangular sets $\mathcal R_1$ and $\mathcal R_2$ where they
meet.}
\end{figure}

As a second step, we give an ``orientation'' to ${\mathcal R}_i$ (for $i=1,2$)
by selecting, in the boundary, two disjoint compact arcs
(i.e. homeomorphic images of the compact unit interval of the real line)
that we denote by ${\mathcal R}^-_{i,\,{ left}}$ and
${\mathcal R}^-_{i,\,{ right}}$ and call the \textit{left} and the \textit{right} sides
of $\partial{\mathcal R}_i\,.$ We also set
$${\mathcal R}^-_{i}\,:={\mathcal R}^-_{i,\,{ left}}\cup {\mathcal R}^-_{i,\,{right}}\,.$$
The closure of $\partial{\mathcal R}_i \setminus {\mathcal
R}^-_{i}$ is denoted by ${\mathcal R}_i^+\,.$ It consists of two
disjoint components which are compact arcs that we name the
\textit{down} and \textit{up} sides of $\partial{\mathcal R}_i\,$
(for a precise definition of \textit{oriented rectangle}, see
Section \ref{sec-3}). In the specific example
of Figure 2, we orientate ${\mathcal R}_1$ and ${\mathcal R}_2$ as
follows. We take as ${\mathcal R}_1^-$ the intersection of
${\mathcal R}_1$ with the inner and outer boundaries of ${\mathcal
A}_P$ and as ${\mathcal R}_2^-$ the intersection of ${\mathcal
R}_2$ with the inner and outer boundaries of ${\mathcal A}_Q\,.$
The way in which we choose to name (as left/right) the two
components of ${\mathcal R}_i^-$ is inessential for the rest of
the discussion. Just to fix the ideas, let's say that we call
``left'' the component of ${\mathcal R}_1^-$ which is closer to
$P$ and the component of ${\mathcal R}_2^-$ which is closer to the
equilibrium point $Q$ (of course, the ``right'' components will be
the other ones).

\medskip
As a third step, we observe that the Poincar\'{e} map associated to $(E^*)$ can be
decomposed as
$$\phi= \phi_{\mu}\circ \phi_0\,,$$
where $\phi_0$ is the Poincar\'{e} map of system $(E_0)$ on the time--interval $[0,r_0]$
and $\phi_{\mu}$ is the Poincar\'{e} map for $(E_{\mu})$ on the time--interval $[0,r_{\mu}]= [0,T-r_0].$
\\
Consider also a path $\gamma: [0,1]\to {\mathcal R}_1$ with $\gamma(0)\in
{\mathcal R}^-_{1,\,{left}}$ and $\gamma(1)\in
{\mathcal R}^-_{1,\,{right}}\,.$ As we'll see in Section \ref{sec-2},
the points of ${\mathcal R}^-_{1,\,{left}}$ move faster than those belonging to
${\mathcal R}^-_{1,\,{right}}$ under the action of system $(E_0).$ Hence, for a
choice of the first switching time $r_0$ large enough, it is possible to
make the path
$$[0,1]\ni \theta\mapsto \phi_0(\gamma(\theta))$$
turn in a spiral--like fashion inside the annulus ${\mathcal A}_P$
and cross at least twice the rectangular region ${\mathcal R}_2$ from
${\mathcal R}^-_{2,\,{left}}$ to ${\mathcal R}^-_{2,\,{right}}\,.$
Thus we can select two sub--intervals of $[0,1]$ such that $\phi_0\circ \gamma$
restricted to each of these intervals is a path contained in ${\mathcal R}_2$
which connects the two components of ${\mathcal R}_2^-\,.$
\\
Now, we observe that the points of ${\mathcal R}^-_{2,\,{left}}$ move faster than those belonging to
${\mathcal R}^-_{2,\,{right}}$ under the action of system $(E_{\mu}).$
Therefore, we can repeat the same argument as above and conclude that for a suitable
choice of $r_{\mu}=T-r_0$ large enough we can transform, via $\phi_{\mu}\,,$ any path in ${\mathcal R}_2\,$
joining the two components of ${\mathcal R}_2^-\,$ onto a path which crosses at least once
${\mathcal R}_1$ from ${\mathcal R}^-_{1,\,{left}}$ to ${\mathcal R}^-_{1,\,{right}}\,.$

\medskip

As a final step, we complete our proof about the existence of chaotic--like dynamics
by applying a topological lemma that we recall
in Section \ref{sec-3} as Lemma \ref{lem-4.1} for the reader's convenience.

In conclusion, our main result
can be stated as follows.

\begin{theorem}\label{th-1.2}
For any choice of positive constants $a,b,c,d, \mu$ with $\mu < a$
and for every pair $({\mathcal A}_P,{\mathcal A}_Q)$ of annuli linked together,
the following conclusion holds:
\\
For every integer $m\geq 2$
there exist two positive constants $\alpha$ and $\beta,$ such that
for each
$$r_0 > \alpha\quad\mbox{ and }\; \; r_{\mu} > \beta,$$
the Poincar\'{e} map associated to system $(E^*)$ induces chaotic
dynamics on $m$ symbols in ${\mathcal R}_1$ and in ${\mathcal
R}_2\,.$
\end{theorem}

In view of our result, one could conclude that
complex dynamics were already hidden in Volterra's work \cite{Vob-31},
since linked twist maps (on long periods) appear as a consequence of the monotonicity of the
period map and of two Volterra's principles:
\textit{
\begin{itemize}
\item{}
``Se si cerca di distruggere uniformemente e proporzionalmente al loro numero gli individui delle due specie,
cresce la media del numero di individui della specie mangiata e
diminuisce quella degli individui della specie mangiante'';
\\
{\footnotesize{\rm(If one tries to destroy uniformly and proportionally to their number the individuals of the
two species, the average of the number of individuals of the eaten species increases and the one of the eating species decreases);}}
\smallskip
\item{}
``Se si distruggono contemporaneamente e uniformemente
individui delle due specie, cresce il rapporto dell'ampiezza della fluttuazione della specie mangiata
all'ampiezza della fluttuazione della specie mangiante''.
\\
{\footnotesize{\rm(If one destroys at the same time and uniformly the individuals of the
two species, the ratio between the amplitude of the fluctuation of the eaten species and
the amplitude of the fluctuation of the eating species increases).}}
\end{itemize}
}
\noindent
Indeed, the first principle says that the position of the center, around which the annuli
can be constructed, varies according to the strength of the
fishing, while
the second principle implies that the shapes of the annuli are suitable for a linking (see \cite[fig. 6]{Vob-31}). The twist conditions
on the boundaries for the associated Poincar\'{e} maps come from an analysis on the periods of the orbits
(which was carried on by Volterra in the limit case, i.e. for orbits near the equilibrium points).

\medskip
In order to clarify the meaning of Theorem \ref{th-1.2}, we introduce the precise concept
of chaotic dynamics that we consider in this work. Our definition
is a modification of the corresponding one in \cite{KiSt-89} and abstracts the
usual interpretation of chaos as the possibility of realizing any coin--flipping sequence,
by giving also a special emphasis to the presence of periodic orbits.
Definitions presenting similar features
have been considered by several authors dealing with nonautonomous ODEs
with periodic coefficients \cite{CaDaPa-02,SrWo-97,WoZg-00}, as well as in abstract theorems
about periodic points and chaotic--like dynamics in metric spaces \cite{Sr-00}.
We refer to Section \ref{sec-3} for a more detailed discussion about the kind
of complex dynamics involved in Definition \ref{def-1.1} below.

\begin{definition}\label{def-1.1}
{\rm
Let $Z$ be a metric space and $\psi: Z\supseteq D_{\psi}\to Z$ be a map. Let also
${\mathcal D}\subseteq D_{\psi}$ be a nonempty set and $m\geq 2$ be an integer.
We say that \textit{$\psi$ induces chaotic dynamics on $m$ symbols
in ${\mathcal D}$} if there exist $m$ pairwise disjoint (nonempty) compact sets
$${\mathcal K}_1\,, {\mathcal K}_2\,,\dots, {\mathcal K}_m\,\subseteq {\mathcal D},$$
such that, for each two--sided sequence of symbols
$$(s_i)_{i\in {\mathbb Z}}\,\in \Sigma_m\,:=\{1,\dots,m\}^{\mathbb Z}\,,$$
there exists a two--sided sequence of points
$$
(w_i)_{i\in {\mathbb Z}}\,\in {\mathcal D}^{\mathbb Z},$$
with
\begin{equation}\label{eq-1.ch1}
w_i \,\in {\mathcal K}_{s_i}\quad \mbox{and } \; w_{i+1} = \psi(w_i),\qquad\forall\, i\in {\mathbb Z}.
\end{equation}
Moreover, if $(s_i)_{i\in {\mathbb Z}}$ is a $k$--periodic
sequence (that is, $s_{i+k} = s_i\,,\forall\, i\in {\mathbb Z}$)
for some $k\geq 1,$ then there exists a $k$--periodic sequence
$(w_i)_{i\in {\mathbb Z}}$ satisfying \eqref{eq-1.ch1}. When we
wish to put in evidence the role of the sets $\mathcal K_j$'s,
we'll also say that \textit{$\psi$ induces chaotic dynamics on $m$
symbols in the set $\mathcal D$ relatively to $(\mathcal
K_1,\dots,\mathcal K_m)$}. }
\end{definition}

\noindent From this definition
it follows that if $\psi$ is \textit{continuous} and \textit{injective}
\footnote{Such assumptions are fulfilled in our application: indeed, in Theorem \ref{th-1.2},
$\psi = \phi$ is a homeomorphism, being the Poincar\'{e} map associated to $(E^*).$}
on the set
$${\mathcal K}:= \bigcup_{j=1}^{m} {\mathcal K}_j\,\subseteq {\mathcal D},$$
then there exists a nonempty compact set
$${\mathcal I}\subseteq {\mathcal K},$$
for which the following properties are fulfilled:
\begin{itemize}
\item{}\;
${\mathcal I}$
is invariant for $\psi$ (i.e. $\psi({\mathcal I}) = {\mathcal I}$);
\item{}\;
$\psi|_{\mathcal I}$ is semiconjugate to the Bernoulli shift on $m$ symbols, that is,
there exists a continuous map $g$ of ${\mathcal I}$ onto $\Sigma_m$ such that the diagram
\begin{equation}\label{diag-1}
\begin{diagram}
\node{{\mathcal I}} \arrow{e,t}{\psi} \arrow{s,l}{g}
      \node{{\mathcal I}} \arrow{s,r}{g} \\
\node{\Sigma_m} \arrow{e,b}{\sigma}
   \node{\Sigma_m}
\end{diagram}
\end{equation}
commutes, where $\sigma$ is the Bernoulli shift on $m$ symbols (i.e. $\sigma:{\Sigma_m}\to{\Sigma_m}$ is the homeomorphism
defined by $\sigma((s_i)_i):=(s_{i+1})_i,\,\forall i\in\mathbb Z);$
\item{}\;
The counterimage $g^{-1}(\textbf{s})\subseteq {\mathcal I}$ of
every
$k$--periodic sequence $\textbf{s} = (s_i)_{i\in {\mathbb Z}}\in \Sigma_m$
contains at least one $k$--periodic point.
\end{itemize}

\medskip
\noindent
For a proof, see Lemma \ref{lem-4.2} in Section \ref{sec-3}. From the above properties it also follows that
$$h_{\rm top}(\psi|_{\mathcal I})\geq h_{\rm top}(\sigma) = \log(m),$$
where $h_{\rm top}$ is the topological entropy \cite{Wa-82}.
Moreover, according to \cite[Lemma 4]{KeKoYo-a01}, there exists a (nonempty) compact invariant set
${\mathcal I }\,'\subseteq {\mathcal I}$ such that
$\psi|_{{\mathcal I }\,'}$ has sensitive dependence on initial conditions,
i.e. $\exists \, \delta > 0:$ $\forall w\in {\mathcal I }\,'$ there is a sequence $w_i$
of points in ${\mathcal I }\,'$ such that $w_i\to w$ and
for each $i\in\mathbb N$ there exists $m=m(i)$ with $\dist(\psi^m(w_i), \psi^m(w))\ge\delta\,.$

\medskip

As remarked in \cite{PaPiZa->}, if we look at Definition \ref{def-1.1} and its consequences
in the context of
concrete examples of ODEs (for instance when $\psi$ turns out to be the Poincar\'{e} map),
condition \eqref{eq-1.ch1} may be sometimes interpreted in terms of the oscillatory behavior of the
solutions. Such situation occurred in \cite{PaZa-a04a, PaZa-08} and takes place
also for system $(E^*).$
Indeed,
from the proof of Theorem \ref{th-1.2}, one sees that it is possible to provide
more precise conclusions in the statement of our main result.
Namely, the following additional properties can be proved:
\\
\textit{
For every splitting of the integer $m\geq 2$ as
$$m= m_1\, m_2,\quad\mbox{with } \; m_1\,, m_2\, \in {\mathbb N}\,,$$
there exist nonnegative integers $\kappa_1\,, \kappa_2\,,$
with $\kappa_1 = \kappa_1(r_0,m_1)$ and $\kappa_2 = \kappa_2(r_{\mu},m_2),$
such that, for each two--sided sequence of symbols
$${\bf{s}} = (s_i)_{i\in{\mathbb Z}}= (p_i,q_i)_{i\in{\mathbb Z}}\in
\{1,\dots,m_1\}^{{\mathbb Z}}\times\{1,\dots,m_2\}^{{\mathbb Z}}\,,$$
there exists a solution
$$\zeta_{\bf{s}}(\cdot) = \bigl( x_{\bf{s}}(\cdot),y_{\bf{s}}(\cdot) \bigr)$$
of $(E^*)$ with
$\zeta_{\bf{s}}(0)\in {\mathcal R}_1$ such that $\zeta_{\bf{s}}(t)$ crosses  ${\mathcal R}_2$
exactly $\kappa_1 + p_i$ times for $t\in \,]iT,r_0 + iT [$ and crosses
${\mathcal R}_1$ exactly $\kappa_2 + q_i$ times for $t\in \,]r_0 + iT,(i+1)T[\,.$
Moreover, if $(s_i)_{i\in{\mathbb Z}}= (p_i,q_i)_{i\in{\mathbb Z}}$ is a periodic sequence,
i.e. $s_{i + k} = s_i\,$ for some $k\geq 1,$ then $\zeta_{\bf{s}}(t + k T) = \zeta_{\bf{s}}(t),$
$\forall\,t\in {\mathbb R}.$}

\smallskip

\noindent
Note that also the splitting $m= m_1\,m_2$ with $m_1 =m$ and $m_2=1$
(or $m_1=1$ and $m_2 =m$) is allowed.
For the precise connection between the role of the compact sets ${\mathcal K}_i$'s in Definition \ref{def-1.1}
and the oscillatory behavior of the solutions, we refer to the proof of Theorem \ref{th-1.2} in the next section.

\noindent
The constants $\alpha$ and $\beta$ which represent the lower bounds for $r_0$ and $r_{\mu}$
can be estimated in terms of $m_1$ and $m_2$ and other geometric parameters, like the
fundamental periods of the orbits bounding the linked annuli (see \eqref{alpha} and \eqref{beta}).

\bigskip
We end this introductory section with a few observations about our main result.

\smallskip

First of all, we note that, according to Theorem \ref{th-1.2},
there is an abundance of chaotic regimes for system $(E^*),$
provided that the time--interval lengths $r_0$ and $r_{\mu}$ (and,
consequently, the period $T$) are sufficiently large. Indeed, we
are able to prove the existence of chaotic invariant sets inside
each intersection of two annular regions linked together. One
could conjecture the presence of Smale horseshoes contained in
such intersections, like in the classical case of the linked twist
maps with circular annuli as domains \cite{De-78}. On the other hand, in our
approach, which is purely topological (like similar ones proposed
in \cite{GaZg-98,KeYo-01,Zg-96}), we just have to check a twist
hypothesis on the boundary, without the need of verifying any
hyperbolicity condition. Hence, our technique
allows to detect the presence of chaotic features by means of
elementary tools. This, of course, does not prevent the
possibility of a further deeper analysis using more complex
computations. We also observe that our result is stable with
respect to small perturbations of the coefficients. Indeed, as it
will be clear from the proof, whenever $r_0 > \alpha$ and
$r_{\mu} > \beta$ are chosen in order to achieve the conclusion of
Theorem \ref{th-1.2}, it follows that there exists a constant
$\varepsilon > 0$ such that Theorem \ref{th-1.2} applies to
equation $(E)$ too, provided that
$$
\int_0^T |a(t) - {\hat{a}}_{\mu}(t)|\,dt < \varepsilon,\quad
\int_0^T |c(t) - {\hat{c}}_{\mu}(t)|\,dt < \varepsilon,
$$
$$
\int_0^T |b(t) - b|\,dt < \varepsilon,\quad
\int_0^T |d(t) - d|\,dt < \varepsilon.
$$
Here, the $T$--periodic coefficients may be in $L^1([0,T])$ or even continuous or smooth
functions (possibly of class $C^{\infty}$).

\medskip

A further remark concerns the fact that, in our model, we have assumed that
the harvesting period starts and ends for both the species at the same moment.
With this respect one could face a more general situation
where some phase--shift between the two harvesting intervals occurs.
Such cases have been already explored in some biological models,
mostly from a numerical
point of view. See \cite{Na-86} for an example dealing with competing species
and \cite{RiMu-93} for a predator--prey system.
If we assume a phase--shift in the periodic coefficients, that is, if we
consider
$$a(t):= {\hat{a}}_{\mu}(t - \theta_1)\quad\mbox{ and } \;c(t):= {\hat{c}}_{\mu}(t - \theta_2),$$
for some $0<\theta_1,\theta_2<T,$
and we also suppose that the length $r_0$ of the time--intervals without harvesting
may be different for the two species (say $r_0 = r_a\in \, ]0,T[$ in the definition of
${\hat{a}}_{\mu}$ and $r_0 = r_c\in \, ]0,T[$ in the definition of
${\hat{c}}_{\mu}$), then
the geometry of our problem turns out to be a combination of linked twist maps on two, three or four
annuli (which are mutually linked together).
In this manner, we increase the possibility of chaotic configurations, provided that the system
is subject to the different regimes for sufficiently long time. For a pictorial comment, see Figure 3
where all the possible links among four annuli are realized.

\begin{figure}[ht]\label{fig-3}
\centering
\includegraphics[scale=0.3]{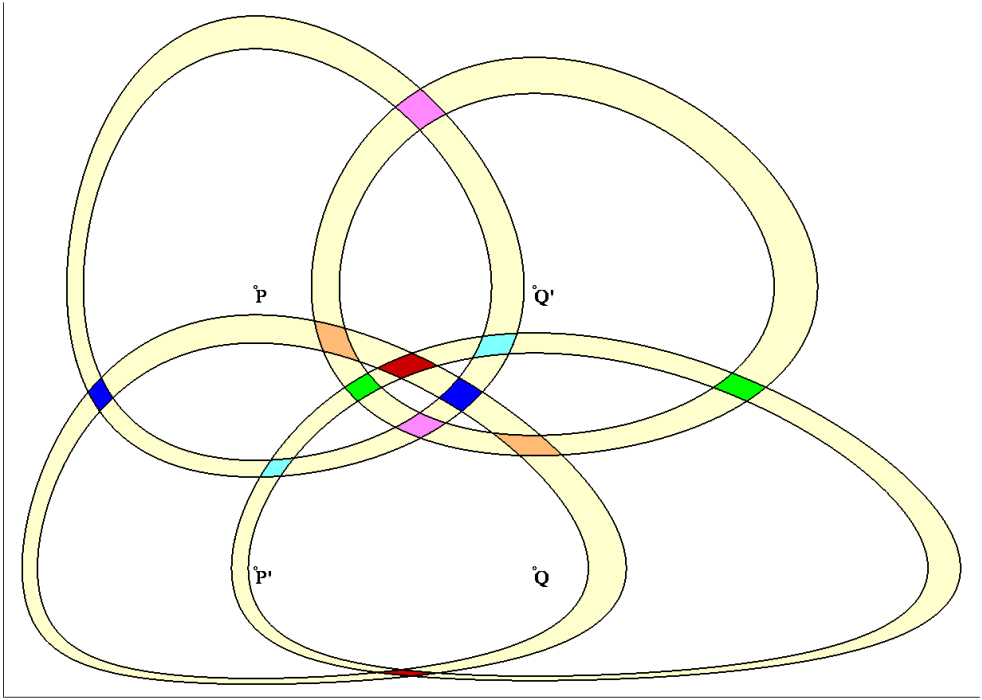}
\caption{\footnotesize We have depicted four linked annular regions bounded by
energy level lines corresponding to Volterra systems with centers at
$P=P_0,$ $P'= (c/d,a_{\mu}/b),$
$Q= P_{\mu}\,$ and $Q'=(c_{\mu}/d,a/b),$
by putting in evidence the regions of mutual intersection,
where it is possible to locate the chaotic invariant sets.}
\end{figure}

\medskip

\noindent
As a final observation, we notice that our approach can be also applied (modulo more technicalities)
to those time--periodic planar Kolmogorov systems \cite{Ko-36}
$$x'= X(t,x,y),\quad y'= Y(t,x,y),$$
which possess dynamical features similar to the ones described above for the Volterra's.
Investigations in this direction will be pursued elsewhere.

\bigskip

The paper is organized as follows. In Section \ref{sec-2} we provide the details for the proof of
Theorem \ref{th-1.2}, in the sense that, following the argument described above, we justify
all the steps by means of some technical estimates.
In Section \ref{sec-3} we recall the topological tools, the corresponding notation
and the abstract theorems which are employed along the paper. We conclude the article with
a brief discussion on chaotic dynamics in the sense of Definition \ref{def-1.1}.

\section{Technical estimates and proof of the main result}\label{sec-2}
Let us consider system $(E_0)$ and let
$$\ell > \chi_0:={\mathcal E}_0(P_0) = \min\{{\mathcal E}_0(x,y): \, x>0, y> 0\}.$$
The level line
$$\Gamma_0(\ell):=\{(x,y)\in ({\mathbb R}_0^+)^2\,:\, {\mathcal E}_0(x,y) = \ell\}$$
is a closed orbit (surrounding $P_0$) which is run in the counterclockwise sense,
completing one turn in a fundamental period that we denote by
$\tau_0(\ell).$
According to classical results on the period of the Lotka--Volterra system
\cite{Ro-85,Wa-86}, we know that the map
$$\tau_0\,: \,]\chi_0,+\infty[\,\to {\mathbb R}$$
is strictly increasing
with $\tau_{0}(+\infty) = +\infty$ and satisfies
$$\lim_{\ell\to \chi_0^+} \tau_0(\ell) = T_0\,:=\frac{2\pi}{\sqrt{a c}}\,.$$
Similarly, if we consider system $(E_{\mu})$ with $0 < \mu < a,$ for
$$h > \chi_{\mu}:={\mathcal E}_{\mu}(P_{\mu}) = \min\{{\mathcal E}_{\mu}(x,y): \, x>0, y> 0\},$$
we denote by $\tau_{\mu}(h)$ the minimal period associated to the orbit
$$\Gamma_{\mu}(h):=\{(x,y)\in ({\mathbb R}_0^+)^2\,:\, {\mathcal E}_{\mu}(x,y) = h\}.$$
Also in this case, we have that the map $h\mapsto \tau_{\mu}(h)$ is strictly increasing
with $\tau_{\mu}(+\infty) = +\infty$ and
$$\lim_{h\to \chi_{\mu}^+} \tau_{\mu}(h) = T_{\mu}\,:=\frac{2\pi}{\sqrt{a_{\mu} c_{\mu}}}\,.$$

\medskip
Before giving the details for the proof of our main result, we
describe conditions on the energy level lines of  two annuli
$\mathcal A_P$ and $\mathcal A_Q,$ centered at $P=P_0= \left(
\frac{c}{d},\frac{a}{b}\right)$ and $Q=P_{\mu}= \left( \frac{c +
\mu}{d},\frac{a - \mu}{b}\right),$ respectively, sufficient to
ensure that they are linked together. With this respect, we have
to consider the intersections among the closed orbits around the
two equilibria and the straight line $r$ passing through the
points $P$ and $Q,$ whose equation is $by+dx-a-c=0.$ We introduce an
orientation on such line by defining an order ``$\preceq$'' among its points.
More precisely, we set $A\preceq B$ (resp. $A\prec B$) if
and only if $x_A \leq x_B$ (resp. $x_A < x_B$), where $A=(x_A,y_A),\, B=(x_B,y_B).$ In this manner,
the order on $r$ is that inherited from the oriented $x$--axis, by projecting the points of $r$
onto the abscissa.
Assume now we
have two closed orbits $\Gamma_0(\ell_1)$ and $\Gamma_0(\ell_2)$
for system $(E_0),$ with $\chi_0<\ell_1<\ell_2.$ Let's call the
intersection points among $r$ and such level lines
$P_{1,-}\,,P_{1,+}\,,$ with reference to $\ell_1,$ and $P_{2,-}\,,P_{2,+}\,,$ with reference to
$\ell_2,$ with
$$P_{2,-}\,\prec\, P_{1,-}\, \prec \,P \prec \,P_{1,+}\, \prec\, P_{2,+}\,.$$
Analogously, when we consider two orbits $\Gamma_{\mu}(h_1)$ and
$\Gamma_{\mu}(h_2)$ for system $(E_{\mu}),$ with
$\chi_{\mu}<h_1<h_2,$ we name the intersection points among $r$ and
these level lines $Q_{1,-}\,,Q_{1,+},$ with reference to $h_1,$ and
$Q_{2,-}\,,Q_{2,+},$ with reference to $h_2,$ with
$$Q_{2,-}\,\prec\, Q_{1,-}\, \prec\, Q \prec\, Q_{1,+}\, \prec \,Q_{2,+}\,.$$
Then the two annuli $\mathcal
A_P$ and $\mathcal A_Q$ turn out to be linked together if
$$P_{2,-}\, \prec\, P_{1,-}\, \preceq\, Q_{2,-}\, \prec\, Q_{1,-}\, \preceq \,P_{1,+}\,
\prec\, P_{2,+}\, \preceq \,Q_{1,+}\, \prec\, Q_{2,+}\,.$$
\smallskip

\noindent {\textit{Proof of Theorem \ref{th-1.2}.}} Consistently
with the notation introduced above, we indicate the level lines
filling $\mathcal A_P$  by $\Gamma_0(\ell),$ with $\ell\in
[\ell_1,\ell_2],$ for some $\chi_0<\ell_1<\ell_2,$ so that
$${\mathcal A}_P=\bigcup_{\ell_1\le\ell\le\ell_2}\Gamma_0(\ell).$$
Analogously,
we'll denote by $\Gamma_{\mu}(h),$ with $h\in [h_1,h_2],$ for some
$\chi_{\mu}<h_1<h_2,$ the level lines filling $\mathcal A_Q,$ so that
we can write
$${\mathcal A}_Q=\bigcup_{h_1\le h\le h_2}\Gamma_{\mu}(h).$$
By construction, such annular regions turn out to be invariant for
the dynamical systems generated by $(E_0)$ and $(E_{\mu}),$
respectively. We consider now the two regions in which each
annulus is cut by the line $r$ (passing through $P$ and $Q$) and
we call such sets $\mathcal A_P^t,\,\mathcal A_P^b,\,\mathcal
A_Q^t$ and $\mathcal A_Q^b,$ in order to have ${\mathcal
A_P}=\mathcal A_P^t\cup \mathcal A_P^b$ and ${\mathcal
A_Q}=\mathcal A_Q^t\cup \mathcal A_Q^b,$ where the sets denoted by
$t$ are the ``upper'' ones and the sets denoted by $b$ are the
``lower'' ones, with respect to the line $r.$ We'll also name
$\mathcal R^b$ the rectangular region in which $\mathcal A_P^b$
and $\mathcal A_Q^b$ meet and analogously we'll denote by
$\mathcal R^t$
the rectangular region belonging to the intersection between $\mathcal A_P^t$ and $\mathcal A_Q^t$
(see Figure 4).

\begin{figure}[ht]\label{fig-4}
\centering
\includegraphics[scale=0.25]{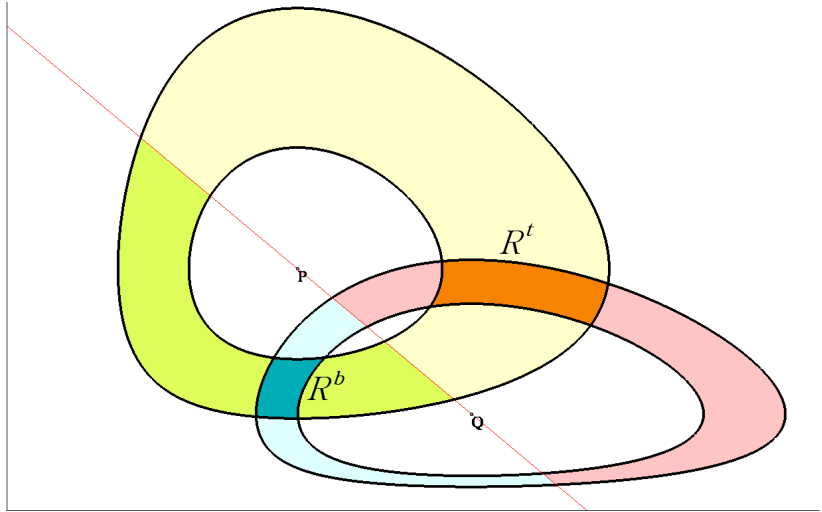}
\caption{\footnotesize For the two linked annuli in the picture (one around
$P$ and the other around $Q$), we have drawn in
different colors the upper and lower
parts (with respect to the dashed line $r$), as well as the intersection regions ${\mathcal R}^t$
and ${\mathcal R}^b$ between them.
As a guideline for the proof, we recall that ${\mathcal A}_P$ is the annulus around $P,$ having as its inner and outer
boundaries the energy level lines $\Gamma_0(\ell_1)$ and $\Gamma_0(\ell_2).$
Similarly, ${\mathcal A}_Q$ is the annulus around $Q,$ having as inner and outer
boundaries the energy level lines $\Gamma_{\mu}(h_1)$ and $\Gamma_{\mu}(h_2).$
}
\end{figure}

\bigskip
\noindent
Let
$$m_1\geq 2\,\quad\mbox{and }\; m_2\geq 1$$
be two fixed integers (the case $m_1 = 1$ and $m_2\geq 2$
can be treated in a similar manner and therefore is omitted).

\noindent
As a first step, we are interested in the solutions of system
$(E_0)$ starting from $\mathcal A_P^b$  and crossing $\mathcal
A_P^t$ at least $m_1$ times. After having performed the
rototraslation of the plane $\mathbb R^2$ that brings the origin
to the point $P$ and makes the $x$--axis coincide with the line
$r,$ whose equations are
\begin{equation*}
\left\{
\begin{array}{ll}
\tilde{x}= (x- \frac{c}{d}) \cos\omega+(y-\frac{a}{b}) \sin\omega\\
\tilde{y}= (\frac{c}{d}-x) \sin\omega+(y-\frac{a}{b}) \cos\omega ,\\
\end{array}
\right.
\end{equation*}
where $\omega:= \arctan(\frac{d}{b}),$ it is possible to use the
Pr\"ufer transformation and introduce generalized polar
coordinates, so that we can express the solution $\zeta(\cdot,z) =
(x(\cdot,z),y(\cdot,z))$ of system $(E_0)$ with initial point in
$z=(x_0,y_0)\in\mathcal A_P^b$ through the radial coordinate
$\rho(t,z)$ and the angular coordinate $\theta(t,z).$ Therefore we
can assume that $\theta(0,z)\in [-\pi,0].$ For any $t\in
[0,r_0]$ and $z\in\mathcal A_P^b,$ let's introduce also the
\textit{rotation number}, that is, the quantity
$$\rot_0(t,z):=\frac{\theta (t,z)-\theta (0,z)}{2\pi}\,,$$
that indicates the normalized angular displacement along the orbit
of system $(E_0)$ starting at $z,$ during the time--interval
$[0,t].$
The continuous dependence of the solutions from the initial data implies that the
function $(t,z)\mapsto \theta(t,z)$ is continuous, as well as
the map $(t,z)\mapsto \rot_0(t,z).$ From the definition of rotation number
and by the star--shapedness of the level lines of ${\mathcal E}_0$
with respect to the point $P,$ we have that for every $z\in \Gamma_0(\ell)$ the following
properties hold:

$$
\begin{array}{ll}
\forall \, j\in {\mathbb Z}\,:\;\; \rot_0(t,z) = j \, &\Longleftrightarrow  \, t = j \, \tau_0(\ell),\\
\forall \, j\in {\mathbb Z}\,:\;\; j < \rot_0(t,z) < j+1 \, &\Longleftrightarrow \, j \,\tau_0(\ell) < t <
(j + 1)\, \tau_0(\ell)
\end{array}
$$
(if the annuli were not star--shaped, the inference `` $\Longleftarrow $ '' would still be true).
Although we have implicitly assumed that $\ell_1 \leq \ell \leq \ell_2\,,$ such properties hold
for every $\ell > \chi_0\,.$

\smallskip

\noindent
Observe that, thanks to the fact that the time--map $\tau_0$ is strictly
increasing, we know that $\tau_0(\ell_1)<\tau_0(\ell_2).$ We shall use this
condition to show that a twist property for the rotation number holds for sufficiently
large time--intervals. Indeed,
we claim that if we choose a switching time $r_0\ge\alpha,$ where
\begin{equation}\label{alpha}
\alpha:=\frac{(m_1+ 3 + \tfrac{1}{2})\,\tau_0(\ell_1)\,\tau_0(\ell_2)}{\tau_0(\ell_2)-\tau_0(\ell_1)}\,,
\end{equation}
then, for any path $\gamma:[0,1]\to {\mathcal A_P},$ with
$\gamma(0)\in\Gamma_0(\ell_1)$ and $\gamma(1)\in\Gamma_0(\ell_2),$ the following interval inclusion is fulfilled:
\begin{equation}\label{eq-3n*}
[\theta(r_0,\gamma(1)),\theta(r_0,\gamma(0))]\supseteq [2\pi
n^*,2\pi(n^*+m_1)-\pi],\,\mbox{for some } n^*=n^*(r_0)\in {\mathbb N}.
\end{equation}
To check our claim, at first we note that for
a path $\gamma(s)$ as above, it holds that
$\rot_0(t,\gamma(0)) \geq \lfloor t/\tau_0(\ell_1)\rfloor$ and
$\rot_0(t,\gamma(1)) \leq \lceil t/\tau_0(\ell_2)\rceil,$ for every $t> 0$
and so
$$\rot_0(t,\gamma(0)) - \rot_0(t,\gamma(1)) > t\, \frac{\tau_0(\ell_2) - \tau_0(\ell_1)}{\tau_0(\ell_1)\,\tau_0(\ell_2)}\, - 2,
\quad
\forall \, t> 0.$$
Hence, for $t\geq \alpha,$ with $\alpha$ defined as in \eqref{alpha}, we obtain
$$\rot_0(t,\gamma(0))> m_1+1+ \tfrac{1}{2} + \rot_0(t,\gamma(1))\,,$$
which, in turns, implies
$$\theta(t,\gamma(0)) - \theta(t,\gamma(1)) > 2 \pi(m_1 + 1),\quad\forall\, t\geq \alpha.$$
Therefore, recalling the bound
$2\pi(\lceil t/\tau_0(\ell_2)\rceil  - \tfrac{3}{2} ) < \theta(t,\gamma(1)) \leq 2 \pi\lceil t/\tau_0(\ell_2)\rceil,$
the interval inclusion \eqref{eq-3n*} is achieved, for
\begin{equation}\label{star}
n^* = n^*(r_0):= \left\lceil \frac{r_0}{\tau_0(\ell_2)}\right\rceil.
\end{equation}
This proves our claim.

\medskip
By the continuity of the composite mapping $[0,1]\ni
s\mapsto \rot_0(r_0,\gamma(s)),$
it follows now that
$$\{\theta
(r_0,\gamma(s)),\,s\in [0,1]\}\supseteq[2\pi
n^*,2\pi(n^*+m_1 -1)+\pi].$$
As a consequence, by the Bolzano theorem, there exist $m_1$ pairwise
disjoint maximal intervals $[t_i{'},t_i{''}]\subseteq [0,1],$
for $i=0,\dots,m_1 - 1,$ such that
$$\theta(r_0,\gamma(s))\in [2\pi n^*+2\pi i,2\pi n^*+\pi+2\pi i],\,\forall s\in [t_i{'},t_i{''}],\, i=0,\dots,m_1 -1,$$
and $\theta(r_0,\gamma(t_i{'}))=2\pi n^*+2\pi i,$ as well as
$\theta(r_0,\gamma(t_i{''}))=2\pi n^*+\pi+2\pi i.$ Setting now
$$\mathcal R_1:=\mathcal R^b\quad\mbox{and } \; \mathcal R_2:=\mathcal R^t,$$
we
orientate such rectangular regions by choosing
$${\mathcal
R}^-_{1,\,{left}}:=\mathcal R_1\cap\Gamma_0(\ell_1)\quad \mbox{and }\;
{\mathcal R}^-_{1,\,{right}}:=\mathcal R_1\cap\Gamma_0(\ell_2),$$
as well as
$${\mathcal R}^-_{2,\,{left}}:=\mathcal
R_2\cap\Gamma_\mu(h_1)\quad \mbox{and } \; {\mathcal
R}^-_{2,\,{right}}:=\mathcal R_2\cap\Gamma_\mu(h_2)$$
(see the caption of Figure 4, as a reminder for the corresponding sets).
\\
Introducing
at last the $m_1$ nonempty and pairwise disjoint compact sets
$$\mathcal H_i:=\{z\in\mathcal A_P^b: \theta(r_0,z)\in [2\pi n^*+2\pi
i,2\pi n^*+\pi+2\pi i]\},\,i=0,\dots,m_1 - 1,$$
we are ready to prove
that
\begin{equation}\label{eq-3.1}
(\mathcal H_i,\phi_0):\widetilde{\mathcal R_1}\stretchx\widetilde{\mathcal R_2},\quad i=0,\dots,m_1-1,
\end{equation}
where we recall that $\phi_0$ is the Poincar\'e map associated to
system $(E_0).$ Indeed, let's take a path $\gamma:[0,1]\to\mathcal
R_1,$  with $\gamma(0)\in {\mathcal R}^-_{1,\,{left}}$ and
$\gamma(1)\in {\mathcal R}^-_{1,\,{right}}.$ For
$r_0\geq\alpha$ and fixing $i\in \{0,\dots,m_1-1\},$
we know that there exists a sub--interval
$[t_i{'},t_i{''}]\subseteq[0,1],$ such that
$\gamma(t)\in\mathcal H_i$ and $\phi_0(\gamma(t))\in \mathcal
A_P^t,\, \forall t\in [t_i{'},t_i{''}].$ Noting now that
$\Gamma_\mu(\phi_0(\gamma(t_i{'})))\le h_1$ and
$\Gamma_\mu(\phi_0(\gamma(t_i{''})))\ge h_2,$ it follows that
there exists a sub--interval
$[t_i^{*},t_i^{**}]\subseteq[t_i{'},t_i{''}]$ such that
$\phi_0(\gamma(t))\in\mathcal R_2,\,\forall t\in
[t_i^{*},t_i^{**}]$ and $\phi_0(\gamma(t_i^{*}))\in{\mathcal
R}^-_{2,\,{left}},$ as well as
$\phi_0(\gamma(t_i^{**}))\in{\mathcal R}^-_{2,\,{right}}.$
Therefore condition (\ref{eq-3.1}) is fulfilled.

\smallskip

Let's turn to system $(E_\mu).$ This time we focus our
attention on the solutions of such system starting from $\mathcal
A_Q^t$  and crossing $\mathcal A_Q^b$ at least $m_2$ times.
Similarly as before, we assume to have performed a rototraslation
of the plane that makes the $x$--axis coincide with the line $r$
and that brings the origin to the point $Q,$ so that we can
express the solution $\zeta(\cdot,w)$ of system $(E_{\mu})$ with
starting point in $w\in\mathcal A_Q^t$ through polar coordinates
$(\tilde\rho,\tilde\theta).$ In particular it holds that
$\tilde\theta(0,w)\in[0,\pi].$ For any $t\in [0,r_\mu]=[0,T-r_0]$
and $w\in\mathcal A_Q^t,$ the rotation number is now defined as
$$\rot_\mu(t,w):=\frac{\tilde\theta (t,w)-\tilde\theta (0,w)}{2\pi}.$$
Since the time--map $\tau_{\mu}$ is strictly increasing, it follows that
$\tau_\mu(h_1)<\tau_\mu(h_2).$ We claim that choosing a swithching time
$r_\mu\ge\beta,$ with
\begin{equation}\label{beta}
\beta:=\frac{(m_2+3 +\tfrac{1}{2})\,\tau_\mu(h_1)\,\tau_\mu(h_2)}{\tau_\mu(h_2)-\tau_\mu(h_1)}\,,
\end{equation}
then, for any path $\sigma:[0,1]\to {\mathcal A_Q},$
with $\sigma(0)\in\Gamma_\mu(h_1)$ and $\sigma(1)\in\Gamma_\mu(h_2),$
the following interval inclusion is satisfied:
\begin{equation}\label{incl}
[\tilde\theta(r_\mu,\sigma(1)),\tilde\theta(r_\mu,\sigma(0))]\supseteq [\pi(
2 n^{**}+1),2\pi(n^{**}+m_2)],\mbox{for some } n^{**}=n^{**}(r_\mu)\in {\mathbb N}.
\end{equation}
The claim can be proved with arguments analogous to the previous ones and therefore its verification is omitted.
The nonnegative integer $n^{**}$ has now to be chosen as
\begin{equation}\label{star2}
n^{**}:=\left\lceil \frac{r_\mu}{\tau_\mu(h_2)}\right\rceil\,.
\end{equation}

By \eqref{incl} and the continuity of the composite mapping $[0,1]\ni
s\mapsto \rot_\mu({r_\mu},\sigma(s)),$  it follows that
$$\{\tilde\theta ({r_\mu},\sigma(s)),\,s\in [0,1]\}\supseteq[2\pi
n^{**}+\pi,2\pi(n^{**}+m_2)].$$

As a consequence, Bolzano theorem ensures the existence of $m_2$ pairwise
disjoint maximal intervals $[s_i{'},s_i{''}]\subseteq [0,1],$
for $i=0,\dots,m_2-1,$ such that
$$\tilde\theta({r_\mu},\sigma(s))\in [2\pi n^{**}+\pi+2\pi i,2\pi n^{**}+2 \pi+2\pi i],\,\forall s\in [s_i{'},s_i{''}],\, i=0,\dots,m_2-1,$$
and $\tilde\theta(r_\mu,\sigma(s_i{'}))=2\pi n^{**}+\pi+2\pi i,$ as well as $\tilde\theta(r_\mu,\sigma(s_i{''}))=2\pi n^{**}+2 \pi+2\pi i.$
\\
For $\widetilde{\mathcal R_1}$ and $\widetilde{\mathcal R_2}$ as above and introducing the $m_2$ nonempty,
compact and pairwise disjoint sets
$$\mathcal K_i:=\{w\in\mathcal A_Q^t: \tilde\theta(r_\mu,w)\in [2\pi n^{**}+\pi+2\pi i,2\pi n^{**}+2 \pi+2\pi i]\},\,i=0,\dots,m_2-1,$$
we are in position to check that
\begin{equation}\label{eq-3.2}
(\mathcal K_i,\phi_\mu):\widetilde{\mathcal R_2}\stretchx\widetilde{\mathcal R_1},\,i=0,\dots,m_2-1,
\end{equation}
where $\phi_\mu$ is the Poincar\'e map associated to system $(E_\mu).$
Indeed, taking a path $\sigma:[0,1]\to\mathcal R_2,$ with $\sigma(0)\in {\mathcal R}^-_{2,\,{left}}$
and $\sigma(1)\in {\mathcal R}^-_{2,\,{right}},$ for $r_\mu\ge\beta$ and for any $i\in \{0,\dots,m_2 -1\}$
fixed, there exists a sub--interval $[s_i{'},s_i{''}]\subseteq[0,1],$ such that $\sigma(t)\in\mathcal K_i$
and $\phi_\mu(\sigma(t))\in \mathcal A_Q^b,\, \forall t\in [s_i{'},s_i{''}].$
Since $\Gamma_0(\phi_\mu(\sigma(s_i{'})))\le\ell_1$ and $\Gamma_0(\phi_\mu(\sigma(s_i{''})))\ge\ell_2,$
there exists a sub--interval $[s_i^{*},s_i^{**}]\subseteq[s_i{'},s_i{''}]$
such that $\phi_\mu(\sigma(t))\in\mathcal R_1,\,\forall t\in [s_i^{*},s_i^{**}]$
and $\phi_\mu(\sigma(s_i^{*}))\in{\mathcal R}^-_{1,\,{left}},$ as well as $\phi_\mu(\sigma(s_i^{**}))\in{\mathcal R}^-_{1,\,{right}}.$
Therefore condition (\ref{eq-3.2}) is proved.

The stretching properties in \eqref{eq-3.1} and \eqref{eq-3.2} allow to apply
Lemma \ref{lem-4.1} and the thesis follows immediately.
\qed

\medskip
\noindent
We observe that in our proof we have
chosen as $\mathcal R_1$ the ``lower'' set $\mathcal R^b$ and as
$\mathcal R_2$ the ``upper'' set $\mathcal R^t.$ However, since the orbits of both systems $(E_0)$ and $(E_\mu)$
are closed, the same argument also works (by slightly modifying some constants, if needed)
if we choose
${\mathcal R}_1 = {\mathcal R}^t$ and
${\mathcal R}_2 = {\mathcal R}^b.$

\section{Topological tools and remarks about chaotic dynamics}\label{sec-3}
In this last section, we briefly recall the topological tools that we have employed along the paper.
They are taken, with minor variants, from \cite{PaPiZa->} and are based on some previous works
\cite{PaZa-a04b,PaZa-a04a,PaZa-08,PiZa->}.
We start with some terminology: \\
Given a metric space $Z,$ we define a \textit{path} in $Z$ as a
continuous map $\gamma: {\mathbb R}\supseteq [0,1]\to Z$
(instead of $[0,1]$ one could take any compact interval $[s_0,s_1]$).
A \textit{sub--path}
$\sigma$ of $\gamma$ is the restriction of $\gamma$ to a
compact sub--interval of its domain. An \textit{arc} is the homeomorphic image of the
compact interval $[0,1].$
By a \textit{generalized rectangle} we mean a set
${\mathcal R}\subseteq Z$ which is homeomorphic to the unit square ${\mathcal Q}:=[0,1]^2\subseteq {\mathbb R}^2.$

If ${\mathcal R}$ is a generalized rectangle and $h: {\mathcal Q}\to h({\mathcal Q})={\mathcal R}$
is a homeomorphism defining it, we call \textit{contour} $\vartheta{\mathcal R}$ of
${\mathcal R}$ the set
$$\vartheta {\mathcal R}:= h(\partial{\mathcal Q}),$$
where $\partial{\mathcal Q}$ is the usual boundary of the unit square. Notice that $\vartheta{\mathcal R}$
is well defined, as it is independent of the choice of the homeomorphism $h.$ We name \textit{oriented rectangle}
the pair
$${\widetilde{\mathcal R}}:=({\mathcal R},{\mathcal R}^-),$$
where ${\mathcal R} \subseteq Z$ is a generalized rectangle and
$${\mathcal R}^- := {\mathcal R}^-_{left}\cup {\mathcal R}^-_{right}$$
is the union of two disjoint compact arcs
${\mathcal R}^-_{left}\,,{\mathcal R}^-_{right}\,\subseteq \vartheta{\mathcal R},$
that we call the \textit{left} and the \textit{right} sides of ${\mathcal R}^-.$
Once that ${\mathcal R}^-$ is fixed, we can also define ${\mathcal R}^+$ as the
closure of $\vartheta {\mathcal R}\setminus {\mathcal R}^-.$ In particular, we set
$${\mathcal R}^+ := {\mathcal R}^+_{down}\cup {\mathcal R}^+_{up}\,,$$
where ${\mathcal R}^+_{down}$ and ${\mathcal R}^+_{up}$ are two disjoint arcs.
In the usual applications, the ambient space $Z$ is just the Euclidean plane and
the generalized rectangles are bounded by some orbit--segments, possibly associated to
different systems.

The central concept in our approach is that of ``stretching along the paths'':

\begin{definition}\label{def-4.1}
{\rm{Suppose that $\psi: Z \supseteq D_{\psi}\to Z$
is a map defined on a set $D_{\psi}$ and let
${\widetilde{X}}:= ({X},{X}^-)$
and
${\widetilde{Y}}:= ({Y},{Y}^-)$
be oriented rectangles of a metric space $Z.$
Let
${\mathcal K}\subseteq {X}\cap D_{\psi}$ be a compact set.
We say that \textit{$({\mathcal K},\psi)$ stretches ${\widetilde{X}}$
to ${\widetilde{Y}}$ along the paths} and write
$$({\mathcal K},\psi): {\widetilde{X}} \stretchx {\widetilde{Y}},$$
if the following conditions hold:
\begin{itemize}
\item{} \; $\psi$ is continuous on ${\mathcal K}\,;$
\item{} \;
for every path $\gamma: [0,1]\to {X}$ such that $\gamma(0)\in {X}^-_{left}$
and $\gamma(1)\in {X}^-_{right}$
(or $\gamma(0)\in {X}^-_{right}$
and $\gamma(1)\in {X}^-_{left}$),
there exists a sub--interval $[t',t'']\subseteq [0,1]$
such that
$$\gamma(t)\in {\mathcal K},\quad \psi(\gamma(t))\in {Y}\,,\;\;
\forall\, t\in [t',t'']$$
and, moreover, $\psi(\gamma(t'))$ and $\psi(\gamma(t''))$ belong to different components of
${Y}^-.$
\end{itemize}

\medskip

\noindent
In the special case in which ${\mathcal K} = X,$
we simply write
$$\psi: {\widetilde{X}} \stretchx {\widetilde{Y}}.$$
}}
\end{definition}

We are now in position to introduce the result about the existence of chaotic dynamics for linked twist maps
that we apply in the proof of Theorem \ref{th-1.2}.
\begin{lemma}\label{lem-4.1}
Let $Z$ be a metric space, let $\Phi: Z\supseteq D_{\Phi}\to Z$ and
$\Psi: Z\supseteq D_{\Psi}\to Z$ be continuous maps and let
${\widetilde{\mathcal R_1}} := ({\mathcal R_1},{\mathcal R_1}^-),$
${\widetilde{\mathcal R_2}} := ({\mathcal R_2},{\mathcal R_2}^-)$ be oriented rectangles in $Z.$
Suppose that the following conditions are satisfied:
\begin{itemize}
\item[$(H_{\Phi})\;\;$] there exist $m_1\geq 1$ pairwise disjoint compact sets
${\mathcal H}_1\,,\dots, {\mathcal H}_{m_1}\,\subseteq {\mathcal R_1}\cap D_{\Phi}$
such that
$\displaystyle{({\mathcal H}_i,\Phi): {\widetilde{\mathcal R_1}}
\stretchx\, {\widetilde{\mathcal R_2}}},$ for $i=1,\dots,m_1\,;$
\\
\item[$(H_{\Psi})\;\;$] there exist $m_2\geq 1$ pairwise disjoint compact sets
${\mathcal K}_1\,,\dots, {\mathcal K}_{m_2}\,\subseteq {\mathcal R_2}\cap D_{\Psi}$
such that
$\displaystyle{({\mathcal K}_i,\Psi): {\widetilde{\mathcal R_2}}
\stretchx\, {\widetilde{\mathcal R_1}}},$ for $i=1,\dots,m_2\,.$
\end{itemize}
If at least one between $m_1$ and $m_2$ is greater or equal than $2,$
then the composite map $\psi:=\Psi\circ \Phi$ induces chaotic dynamics on $m_1 \times m_2$ symbols
in the set
$${\mathcal H}^*:=
\bigcup_{i=1,\dots,m_1\atop j=1,\dots,m_2} {\mathcal H}'_{i,j},\,\quad\mbox{where }\;
{\mathcal H}'_{i,j}\,:= {\mathcal H}_i\cap \Phi^{-1}({\mathcal K}_j).$$
Moreover,
for each sequence of $m_1\times m_2$ symbols
$$\textbf{s}=(s_n)_n=(p_n,q_n)_n\in
\{1,\dots,m_1\}^{\mathbb N}\times  \{1,\dots,m_2\}^{\mathbb N},$$
there exists a compact connected set ${\mathcal C}_{\textbf{s}}\subseteq {\mathcal H}'_{p_0,q_0}$
with
$$
{\mathcal C}_{\textbf{s}}\cap {{\mathcal R}}^+_{1,down}\not=\emptyset,
\quad{\mathcal C}_{\textbf{s}}\cap {{\mathcal R}}^+_{1,up}\not=\emptyset
$$
and such that,
for every $w\in {\mathcal C}_{\textbf{s}}\,,$
there exists a sequence $(y_n)_{n}$ with $y_0 = w$ and
$$y_n \in {\mathcal H}'_{p_n,q_n}\,,\quad \psi(y_n) = y_{n+1}\,,\;\forall\, n\geq 0.$$
\end{lemma}

\noindent
Note that in Theorem \ref{th-1.2}
we haven't used the last part of the previous lemma.
Therefore it would be possible to improve our main result on Lotka--Volterra systems
by adding information about the existence of continua of points for which we have a control on the forward itineraries.

\medskip

We end the paper by trying to clarify the relationship between the concept of chaos expressed
in Definition \ref{def-1.1} and other ones available in the literature, with special reference to the
semiconjugation with the Bernoulli shift and the density of the periodic points.
Before giving the next lemmas we just recall
some basic definitions. We denote by
$\Sigma_m = \{1,\dots, m\}^{\mathbb Z}$
the set
of the two--sided sequences of $m$ symbols. Analogously, by
$\Sigma_m ^+= \{1,\dots, m\}^{\mathbb N}$
we mean the set of one--sided sequences of $m$ symbols (where ${\mathbb N}$ is the set of nonnegative integers).
Introducing the standard distance
\begin{equation}\label{eq-3c.dist}
d(\textbf{s}', \textbf{s}'') := \sum_{i\in {\mathbb I}} \frac{|s'_i - s''_i|}{m^{|i| + 1}}\,,\quad
\mbox{ where }\; \textbf{s}'=(s'_i)_{i\in {\mathbb I}}\,,\,
\textbf{s}''=(s''_i)_{i\in {\mathbb I}}\,,
\end{equation}
for ${\mathbb I} = {\mathbb Z}$ and for ${\mathbb I} = {\mathbb N},$
we make $\Sigma_m$ and $\Sigma_m^+$ compact metric spaces.

\medskip

Our first result is quite classical.
Indeed, the same conclusions of Lemma \ref{lem-4.2}
have been obtained by several different authors dealing with similar situations
(see, for instance, \cite[Theorem 3]{Sr-00}). Nonetheless, we give a detailed proof for
the reader's convenience. We also thank Duccio Papini for useful discussions about this
topic.

\begin{lemma}\label{lem-4.2}
Let $(Z,d_Z)$ be a metric space,  $\psi: Z\supseteq D_{\psi} \to Z$ be a map
which induces chaotic dynamics on $m\geq 2$ symbols in a set ${\mathcal D}\subseteq D_{\psi}$\,, relatively to
$({\mathcal K}_1,\dots,{\mathcal K}_{m})$ (according to Definition \ref{def-1.1}).
Assume also that $\psi$ is continuous and injective on
$${\mathcal K}:=\bigcup_{j=1}^{m} {\mathcal K}_j\,.$$
Define the nonempty compact set
\begin{equation}\label{eq-4.2}
{\mathcal I}:= \bigcap_{j=-\infty}^{+\infty} \psi^{-j}({\mathcal K}).
\end{equation}
Then ${\mathcal I}$ is invariant for $\psi$ and
$\psi|_{\mathcal I}$ is semiconjugate to the two--sided
$m$--shift, through a continuous surjection $g:\mathcal
I\to\Sigma_m$ as in \eqref{diag-1}. Moreover, the counterimage through
$g$ of any $k$--periodic sequence in $\Sigma_m$ contains at least
one $k$--periodic point of ${\mathcal I}.$
\end{lemma}

\begin{proof}
First of all, we observe that
$w\in {\mathcal I}$ if and only if there exists a full orbit, that is, a two--sided itinerary,
$(w_i)_{i\in{\mathbb Z}}$ such that
$w_0=w$ and $\psi(w_{i-1}) = w_i\in {\mathcal K}$ for every $i\in {\mathbb Z}.$
By the assumptions on $\psi$ coming from Definition \ref{def-1.1}
and standard properties of compact sets,
it follows immediately that ${\mathcal I}$ is a nonempty compact set such that
$\psi({\mathcal I}) = {\mathcal I}.$
\\
As a next step, we introduce a function $g_1\,,$ associating to any $w\in {\mathcal I}$ its corresponding full orbit
$$s_w:= (w_i)_{i\in {\mathbb Z}}\,,$$
that is,
the sequence of points of the set ${\mathcal I}$ defined by
$$w_i:= \psi^{i}(w),\quad\forall\, i\in {\mathbb Z},$$
with the usual convention $\psi^0 = \Id_{\mathcal K}$ and $\psi^1 = \psi.$
The injectivity of $\psi$ implies that the map
$$g_1:w\mapsto s_w$$
is well defined.
Recalling that the ${\mathcal K}_j$'s are
pairwise disjoint, we know that for every term $w_i$ of $s_w$ there exists a
unique label
$$s_i = s_i(w_i),\quad\mbox{with } s_i\in \{1,\dots,m\},$$
such that
$w_i \in {\mathcal K}_{s_i}\,.$
Hence, the map
\begin{equation}\label{eq-g2}
g_2: s_w \mapsto (s_i)_{i\in {\mathbb Z}}\in \Sigma_m
\end{equation}
is also well defined. Thus, if we set
$$g:= g_2\circ g_1: {\mathcal I}\to \Sigma_m,$$
by Definition \ref{def-1.1} we obtain a surjective map that makes the diagram
\eqref{diag-1} commute. Moreover, the inverse image through
$g$ of any $k$--periodic sequence in $\Sigma_m$ contains at least
one $k$--periodic point of ${\mathcal I}.$
\\
Finally, we show that $g$ is continuous. To such aim we call
$$\epsilon:= \min_{1 \leq i\not=j\leq m} d_Z({\mathcal K}_i,{\mathcal K}_j) > 0$$
and note that
the map $\psi|_{\mathcal I}: {\mathcal I}\to \psi({\mathcal I})={\mathcal I}$
is a homeomorphism. Therefore the following property holds:
\begin{itemize}
\item[] $\forall\, n\in {\mathbb N},\; \exists\, \delta = \delta_n> 0:\;$  for
$u,v\in {\mathcal I},$ with $s_u:=(u_i)_{i\in{\mathbb Z}}\,,$
$s_v:=(v_i)_{i\in{\mathbb Z}}\,,$
$$d_Z(u,v) < \delta \Longrightarrow d_Z(u_i,v_i) < \epsilon\,,\;\forall\, |i|\leq n
\Longrightarrow s_i(u_i) = s_i(v_i),\;\forall\, |i|\leq n.$$
\end{itemize}

\noindent From this fact and the choice of the distance in $\Sigma_m$
(see \eqref{eq-3c.dist}), the continuity of $g$ is easily proved.
\end{proof}

The next result is just a more precise version of Lemma \ref{lem-4.2},
with reference to the periodic points.

\begin{lemma}\label{lem-4.3}
Under the same assumptions of Lemma \ref{lem-4.2},
there exists a compact invariant set $\Lambda \subseteq {\mathcal I}$
such that
$\psi|_{\Lambda}$ is semiconjugate to the two--sided
$m$--shift, through the continuous surjection $g|_{\Lambda}$
(where
$g:\mathcal
I\to\Sigma_m$ is like in \eqref{diag-1}).
The set of the periodic points of $\psi|_{\mathcal I}$ is dense in $\Lambda$ and,
moreover, the counterimage through
$g$ of any $k$--periodic sequence in $\Sigma_m$ contains at least
one $k$--periodic point of $\Lambda.$
\end{lemma}

\begin{proof}
For ${\mathcal I}$ defined as in \eqref{eq-4.2}, we consider the
subset ${\mathcal P}$ made of the periodic points of $\psi|_{\mathcal I}\,,$
that is,
$${\mathcal P}:= \{w\in {\mathcal I}: \exists \, k\geq 1, \; \psi^{k}(w) = w\}$$
and define
$$\Lambda:= \overline{\mathcal P}.$$
Since $\psi({\mathcal P}) = {\mathcal P},$
it follows that $\psi(\Lambda) = \Lambda.$ From the last statement in Lemma \ref{lem-4.2} we also find that
$g({\mathcal P})$ coincides with the subset of $\Sigma_m$
made by the two--sided periodic sequences of $m$ symbols, which is dense in $\Sigma_m\,.$
This latter fact implies the surjectivity of $g|_{\Lambda}: \Lambda \to \Sigma_m\,.$
The remaining properties are a straightforward consequence of the corresponding ones in Lemma \ref{lem-4.2}.
\end{proof}

\smallskip

In both the above results, we have required the injectivity of the map $\psi.$
This is not a heavy restriction in the applications to ODEs (like those in Section \ref{sec-1}
and in \cite{PaPiZa->,PaZa-08}), where $\psi$ is the Poincar\'{e} map.
Anyway, in the abstract setting of the metric spaces and with reference to Definition \ref{def-1.1},
it is still possible to obtain some suitable versions of Lemma \ref{lem-4.2} and
Lemma \ref{lem-4.3} for a map $\psi$ which is continuous on ${\mathcal K},$
but without the assumption of injectivity. Such task can be accomplished in different ways. A first possibility is
that of considering a semiconjugation with the Bernoulli shift on $m$ symbols for
one--sided sequences. Indeed, any initial point $w$ uniquely determines the
forward sequence $s^+_{w}:=(\psi^i(w))_{i\in {\mathbb N}}\,,$ to which we can associate
a well determined sequence of symbols $(s_i)_{i\in {\mathbb N}}\in \Sigma_m^+$ as in the definition
of $g_2$ in \eqref{eq-g2}.
Another possibility is that of
considering, in place of $g,$ only the map $g_2$, that is,
to associate to any full orbit $(w_i)_{i\in{\mathbb Z}}\,,$ with $\psi(w_i) = w_{i+1}\,,$
the sequence $(s_i)_{i\in {\mathbb Z}}\in \Sigma_m$ such that $w_i\in {\mathcal K}_{s_i}\,.$
This second approach is followed by Lani-Wayda and Srzednicki in \cite{LWSr-02},
where the authors obtain a variant of Lemma \ref{lem-4.3} for a set $\Lambda$ replaced
by the closure (denoted by ${\mathcal T}$) of the set of the terms of the periodic sequences in ${\mathcal K}$
which are full orbits of $\psi$ (and project through $g_2$ onto a periodic sequence of $\Sigma_m$).
Hence, in this case, the commutative diagram \eqref{diag-1} becomes
$$
\begin{diagram}
\node{{\mathcal T}^{\mathbb Z}} \arrow{e,t}{\psi^{\mathbb Z}} \arrow{s,l}{g_2}
      \node{{\mathcal T}^{\mathbb Z}} \arrow{s,r}{g_2} \\
\node{\Sigma_m} \arrow{e,b}{\sigma}
   \node{\Sigma_m}
\end{diagram}
$$
where $\psi^{\mathbb Z}\bigl( (w_i)_{i \in {\mathbb Z}}\bigr ):= (w_{i+1})_{i \in {\mathbb Z}}\,.$

\section{Conclusion}\label{sec-5}
It is a pleasure and a honor to have the possibility to dedicate our work
to the memory of Professor Andrzej Lasota, who gave fundamental contributions both in the area of
chaotic dynamics and  in the study of mathematical models for biological systems.

\nocite{*}
\bibliographystyle{amsplain}
{\footnotesize
\bibliography{piza4}
}

\end{document}